\RequirePackage{fix-cm}
\documentclass[smallextended]{svjour3}       
\smartqed  
\usepackage{graphicx}
\usepackage{mathptmx}      
\usepackage{amsmath,amssymb,stmaryrd}
\usepackage{color}
\usepackage{framed}
\newcommand{\LL}{\mathcal{L}}
\newcommand{\Th}{\mathcal{T}_h}
\newcommand{\Eh}{\mathcal{E}_h}
\DeclareMathOperator{\diam}{diam}
\newcommand{\ovec}{\overrightarrow}
\newcommand{\RR}{\mathbb{R}}
\newcommand{\new}[1]{\textcolor{black}{#1}}
\journalname{Numerische Mathematik}

\begin{document}

\title{A primal discontinuous Galerkin method with static condensation on very general 
meshes
}
\titlerunning{Discontinuous Galerkin method with static condensation }        

\author{
Alexei Lozinski
}


\institute{A. Lozinski \at
              Laboratoire de Math\'{e}matiques de Besan\c{c}on, CNRS UMR 6623, 
              Univ. Bourgogne Franche-Comt\'{e}, 16 route de Gray, 25030 
Besan\c{c}on Cedex, France. \\
              Tel.: +33 3 81 66 63 16\\
              Fax: +33 3 81 66 66 23\\
              \email{alexei.lozinski@univ-fcomte.fr}           
}

\date{Received: date / Accepted: date}

\maketitle

\begin{abstract}
We propose an efficient variant of a primal Discontinuous Galerkin method with interior penalty for the 
second order elliptic equations on very general meshes (polytopes with eventually curved boundaries).   Efficiency, especially when higher order polynomials are used, is achieved by static condensation, i.e. a local 
elimination of certain degrees of freedom cell by cell. This alters the 
original method in a way that preserves the optimal error estimates. Numerical 
experiments confirm that the solutions produced by the new method are indeed 
very close to that produced by the classical one.
\end{abstract}

\section{Introduction}
\label{intro}

The recent years have seen the emergence (or the revival) of several numerical methods capable to solve approximately elliptic partial differential equations using general polygonal/polyhedral meshes. This is witnessed for example by the book \cite{bridges}. The methods reviewed in this book (if we restrict our attention only to finite element type methods using piecewise polynomial approximation spaces in one form or another) include interior penalty discontinuous Galerkin (DG) methods \cite{cangiani14,antonietti16}, hybridizable discontinuous Galerkin (HDG) methods (\cite{cockburn16}, introduced in \cite{cockburn09}),  
the Virtual Element (VE) method (\cite{veiga16}, introduced in \cite{veiga13,brezzi13}), the Hybrid High-Order (HHO) method (\cite{dipietro16}, introduced in \cite{dipietro15,dipietro14}). One can add to this list the weak Galerkin finite element \cite{wang13}, which is similar to HDG. 
The relations between HHO and HDG methods were exhibited in \cite{cockburn_dipietro}.

Among the above, the primal interior penalty DG methods are the most classical. In the symmetric form, also referred to as SIP -- symmetric interior penalty, this method dates back to \cite{wheeler78,arnold82} and is now presented and thoroughly studied in several monographs, for example  \cite{riviere_book,dipietro_book}. It is well suited to the discretization on very general meshes because its approximation space is populated by polynomials of degree, say $\le k$, on each mesh cell without any constraints linking the polynomials on two adjacent cells. It leaves thus a lot of freedom on the choice of the mesh cells which can be not only polytopes but also virtually any geometrical shapes. It is generally admitted however that the SIP method is too expensive especially when higher order polynomials are employed. Indeed, its cost, i.e. the dimension of the approximation space, is the product of the number of mesh cells and the dimension of the space of polynomials of degree $\le k$. The cost on a given mesh is thus proportional to $k^2$ in 2D (resp. $k^3$ in 3D). This should be contrasted with the cost of HDG or HHO methods which is proportional to $k$ in 2D (resp. $k^2$ in 3D).    

The goal of the present article is to modify the SIP method so that its cost is reduced to that of HDG or HHO methods. In doing so, we inspire ourselves from the static condensation procedure for the standard continuous Galerkin (CG) finite element methods. It is indeed well known that the dimension of the approximation space in CG is proportional to $k^2$ in 2D on a given mesh, but the degrees of freedom interior to each mesh cell can be locally eliminated which leaves a global problem of the size proportional to $k$ (these numbers are changed to, respectively, $k^3$ and $k^2$ in 3D). Although the notion of interior degrees of freedom does not make sense in the DG context, we shall be able to select, on each mesh cell, a subspace of the approximating polynomials that can be used to construct a local  approximation through the solution of a local problem. The remaining degrees of freedom will then be used in a global problem. We shall thus achieve a significant reduction of the global problem size in the DG SIP-like method, similarly to that achieved in CG by static condensation. The resulting DG method, which can be refereed to as scSIP (static condensation SIP), will not produce exactly the same approximation as the original SIP method. We shall prove however that these two solutions satisfy the same optimal \textit{a priori} error bounds in $H^1$ and $L^2$ norms. Moreover, they turn out to be very close to each other in our numerical experiments. 

We treat here only the diffusion equation with variable, but sufficiently smooth, coefficients. The extension to other problems, such as convection-reaction-diffusion, linear elasticity, Stokes, as well as to other DG variants (IIP, NIP) seems relatively straight-forward. Our assumptions on the mesh allow for cells of general shape, not necessarily the polytopes. 

The article is organized as follows: in the next section, we present the idea of our method starting by the description of the governing equations. We the recall the static condensation for the classical CG FEM. Our variants of DG FEM (SIP and scSIP) are 
first introduced in Subsection \ref{sec2}.2. The convergence proofs are in Section \ref{sec3}. 
They are done assuming some properties of the discontinuous FE 
spaces and the underlying mesh. In Section \ref{SecMesh}, we give an example of the hypotheses on the mesh under 
which the necessary properties of the FE spaces can be established. Finally, some implementation details and numerical illustrations are presented in Section \ref{sec5}.

\section{Description of the problem and static condensation for FEM (CG and DG cases)}
\label{sec2}

We consider the second-order elliptic problem
\begin{equation}
  \label{Pb} {\LL u} = f \text{ in } \Omega , \quad u = g \text{ on } \partial \Omega
\end{equation}
where $\Omega \subset \mathbb{R}^d$, $d = 2$ or 3, is a bounded Lipschitz 
domain, $f \in L^2 (\Omega)$ and $g \in H^{1 / 2} (\partial \Omega)$ are given
functions. The differential operator $\LL$ is defined by
\[ {\LL u} = - \partial_i (A_{{ij}} (x) \partial_j u) \]
with $\partial_i$ denoting the partial derivative in the direction $x_i$, $i =
1, \ldots, d$ and assuming the summation over $i,j$. The coefficients $A_{{ij}}$ are supposed to form a
positive definite matrix $A=A(x)$ for any $x \in \Omega$ which is sufficiently smooth
with respect to $x$ so that
\begin{equation}
  \label{regel1} \alpha | \xi |^2 \leq \xi^T A (x) \xi \leq \beta |
  \xi |^2, \hspace{1em} \forall \xi \in \mathbb{R}^d, x \in \Omega
\end{equation}
and
\begin{equation}
  \label{regel2} | \nabla A_{{ij}} (x) | \leq M, \hspace{1em}
  \forall x \in \Omega, i, j = 1, \ldots, d
\end{equation}
with some constants $\beta \geq \alpha > 0$, $M > 0$.

\subsection{Static condensation for CG FEM}
To present our idea, we start by recalling the idea of static condensation, 
going back to \cite{guyan65}, as applied to the usual
CG finite element method for problem (\ref{Pb}).
Let us assume for the moment (in this subsection only) that $\Omega$ is a polygon (polyhedron) and
introduce a conforming mesh $\mathcal{T}_h$ on $\Omega$ consisting of
triangles (tetrahedrons). Assuming for simplicity $g = 0$, the usual continuous
$\mathbb{P}_k$ finite element discretization of (\ref{Pb}) is then written:
find $u_h \in W_h$ such that
\begin{equation}
  \label{PbPkc} a (u_h, v_h) : = \int_{\Omega} A \nabla u_h \cdot \nabla v_h =
  \int_{\Omega} fv_h, \quad \forall v_h \in W_h
\end{equation}
where $W_h$ is the space of continuous piecewise polynomial functions
(polynomials of degree $\leq k$ on each mesh cell $T \in \mathcal{T}_h$
for some $k \geq 1$) vanishing on $\partial \Omega$. The size of this
problem, i.e. the dimension of $W_h$, is of order $k^2$ on a given mesh in 2D
(resp. $k^3$ in 3D). To reduce this cost, one can decompose the space $W_h$ as
follows
\[ W_h = W_h^{{loc}} \oplus W'_h  \]
where the subspace $W_h^{{loc}}$ consists of functions of $W_h$ that
vanish on the boundaries of all the mesh cells $T \in \mathcal{T}_h$, and
$W'_h$ is the complement of $W_h^{{loc}}$, orthogonal with respect to the
bilinear form $a$. Decomposing $u_h = u_h^{{loc}} + u_h'$ with
$u_h^{{loc}} \in W_h^{{loc}}$ and $u'_h  \in W_h'$ we see that
(\ref{PbPkc}) is split into two independent problems
\begin{align}
  \label{Pbloc} u_h^{{loc}} \in W_h^{{loc}} &: & a
  (u_h^{{loc}}, v_h^{{loc}}) &= \int_{\Omega}
  {fv}^{{loc}}_h, \hspace{1em} \forall v^{{loc}}_h \in
  W^{{loc}}_h
\\
  \label{Pbglob} u_h' \in W_h' &: & a (u_h', v_h') &= \int_{\Omega}
  {fv}'_h, \hspace{1em} \forall v'_h \in W'_h
\end{align}
The first problem above is further split into a collection of
mutually independent problems on every mesh cell $T \in \mathcal{T}_h$:
\begin{equation}\label{PblocT}
 \text{Find }u^{{loc,T}}_h \text{ such that }
 \int_T A \nabla u^{{loc,T}}_h \cdot \nabla v^{{loc}, T}_h =
\int_{\Omega} {fv}^{{loc}, T}_h, \quad \forall v^{{loc}, T}_h \in
W^{{loc}, T}_h 
\end{equation}
where $W^{{loc}, T}_h$ is the restriction of $W^{{loc}}_h$ on $T$,
i.e. the set of all polynomials of degree $\leq k$
vanishing on $\partial T$. The cost of solution of these local problems is
negligible and we thus get very cheaply $u^{{loc}}_h|_T=u^{{loc}, T}_h$. Note 
also that Problem  (\ref{PblocT}) can be recast as
\begin{equation}
  \label{piloc} \pi_T \LL(u^{{loc}}_h |_T) = \pi_T f ,
  \hspace{1em} \forall T \in \mathcal{T}_h
\end{equation}
where $\pi_T$ is the projection to $W^{{loc}, T}_h$, orthogonal in $L^2
(T)$. 

On the other hand, Problem (\ref{Pbglob}) remains global but its
size is only proportional to $k$ in 2D (resp. $k^2$ in 3D) which is much
smaller than that of the original problem (\ref{PbPkc}). Indeed, the degrees of 
freedom are associated to the standard interpolation points 
of $\mathbb{P}^k$ finite elements on the edges of the mesh. Note also that a
basis for $W_h'$ can be constructed solving cheap local problems of the type
\begin{equation}
  \label{piglob} \pi_T \LL(v _h' |_T) = 0 , \hspace{1em}
  \forall T \in \mathcal{T}_h
\end{equation}
with appropriate boundary conditions on $\partial T$ insuring the continuity
of functions in $W_h'$.

\subsection{DG FEM: SIP and scSIP methods}
We turn now to the main subject of this paper: the DG methods. We now let $\Omega$ be a
bounded domain of general shape, and $\mathcal{T}_h$ be a
splitting of $\Omega$ into a collection of non-overlapping subdomains (again
of general shape, the precise definitions and assumptions on the mesh will be 
given Sections \ref{sec3} and \ref{SecMesh}). Let $V_h$ denote the space of
discontinuous piecewise polynomial functions of degree $\leq
k$ on each mesh cell $T \in \mathcal{T}_h$ for some $k \geq 2$:\footnote{\new{The usual SIP DG method makes perfect sense also for piecewise linear polynomials ($k=1$). We restrict ourselves however to $k\ge 2$ since the forthcoming modification of the method allowing for the static condensation is pertinent to this case only.}}
\begin{equation}\label{spacek} 
  V_h = \{v \in L^2 (\Omega) : v|_T \in \mathbb{P}_k (T), \forall T \in 
\mathcal{T}_h \}
\end{equation}
The SIP (symmetric interior penalty) method is then written as: find $u_h \in V_h$ such that
\begin{equation}
  \label{PbDG} a_h (u_h, v_h) = L_h (v_h), \hspace{1em} \forall v_h \in V_h
\end{equation}
with the bilinear form $a_h$ and the linear form $L_h$ defined by 
\begin{multline}\label{ah} 
 a_h (u, v) = \sum_{T \in \mathcal{T}_h} \int_T A \nabla u \cdot
  \nabla v - \sum_{E \in \mathcal{E}_h} \int_E (\{A \nabla u \cdot n\}[v] +\{A
  \nabla v \cdot n\}[u]) \\
  + \sum_{E \in \mathcal{E}_h} \frac{\gamma}{h_E} \int_E [u] [v]
\end{multline}
and 
\begin{equation}
  \label{Lh} L_h (v) = \sum_{T \in \mathcal{T}_h} \int_T fv + \sum_{E \in
  \mathcal{E}_h^b}  \int_E g \left( \frac{\gamma}{h_E} v - A \nabla v \cdot n
  \right)
\end{equation}
where $\mathcal{E}_h$ is the set of all the edges/faces of the mesh, 
$\mathcal{E}_h^b\subset\mathcal{E}_h$ regroups the edges/faces on the boundary
$\partial\Omega$, $n$, $[\cdot]$ and $\{\cdot\}$ denote the unit normal, the jump and the mean over 
$E\in\mathcal{E}_h$. More precisely, for any internal facet $E$ shared by two mesh cells $T_1^E$ and $T_2^E$, we choose $n$ as the unit vector, normal to $E$ and looking from $T_1^E$ to $T_2^E$. We then define for any function $v$ which is $H^1$ on both $T_1^E$ and $T_2^E$ but discontinuous across $E$
$$
[v]|_E:=v|_{T_1^E} - v|_{T_2^E},
\quad
\{A\nabla v \cdot n\}|_E:=\frac 12 A \left(\nabla v|_{T_1^E}+\nabla v|_{T_2^E}\right) \cdot n
$$
On a boundary edge $E\in\mathcal{E}_h^b$, $n$ is the unit normal looking outward $\Omega$ and $[v]=v$,  $\{A\nabla v \cdot n\}=A\nabla v \cdot n$. \new{The parameter $h_E$ in (\ref{ah})--(\ref{Lh}) is the local length scale of the mesh near the facet $E$ which will be properly defined in Lemma \ref{lemma1}.\footnote{\new{The usual choice $h_E=\diam(E)$ is not appropriate on general meshes since some of the facets can be of much smaller diameter than that of the adjacent cell.}} Finally, $\gamma$ is the interior penalty parameter which should be chosen sufficiently big.}  

Unlike the case of continuous finite elements, Problem (\ref{PbDG}) does not 
allow directly for a static condensation. However, we can construct a modification of (\ref{PbDG}) that mimics the
characterization of local and global components of the solution by the
projectors on local polynomial spaces (\ref{piloc})--(\ref{piglob}). These
spaces are now defined simply as $$V^{{loc}, T}_h =\mathbb{P}^{k -
2}(T)$$ We also let $\pi_{T, k - 2}$ to be the projection to $V^{{loc}, T}_h$, 
orthogonal   in $L^2 (T)$, and propose the following scheme:
\begin{framed}
\begin{itemize}
  \item Compute $u_h^{{loc}} \in V_h$ by solving
  \begin{equation}
    \label{pilocDG} \pi_{T, k - 2} \LL(u^{{loc}}_h |_T) = \pi_{T, k -
    2} f , \hspace{1em} \forall T \in \mathcal{T}_h
  \end{equation}
\new{i.e. find $u^{{loc}}_h |_T \in\mathbb{P}^{k}(T)$ on all mesh cells $T\in\mathcal{T}_h$ such that 
$$ \int_T \LL(u^{{loc}}_h |_T)q_T=\int_T fq_T,\quad\forall q_T\in\mathbb{P}^{k -
2}(T)
$$ }	
  \item Define the subspace of $V_h$
  \begin{equation}
    \label{piglobDG} V_h' = \left\{ v_h' \in V_h : \pi_{T, k - 2}
    \LL(v _h' |_T) = 0 , \hspace{1em} \forall T \in
    \mathcal{T}_h \right\}
  \end{equation}
  \new{i.e. the subspace of functions $v_h' \in V_h$ such that 
	\begin{equation}
    \label{piglobDGprob} \int_T \LL(v'_h |_T)q_T=0, \quad\forall q_T\in\mathbb{P}^{k -2}(T)
	\end{equation}	on all mesh cells $T\in\mathcal{T}_h$. }	
\item Compute $u_h' \in V_h'$ such that
  \begin{equation}
    \label{PbglobDG} a_h (u_h', v_h') = L_h (v'_h) - a_h (u_h^{{loc}},
    v_h'), \hspace{1em} \forall v'_h \in V'_h
  \end{equation}
  \item Set \begin{equation}
    \label{sumDG} 
    u_h = u_h^{{loc}} + u_h'.
    \end{equation}
\end{itemize}
\end{framed}

  \new{We shall show that the local problem (\ref{pilocDG}) admits an infinity of solutions. We can choose 
 any of these solutions on each mesh cell to form $u_h^{loc}$. Nevertheless, the final result $u_h$ given by (\ref{sumDG}) is unique, cf. Lemma \ref{lemma3}. }

\new{Note that the dimension of the ``global'' space $V_h'$ on a given mesh 
is of order $k$ in 2D ($k^2$ in 3D) so that global problem (\ref{PbglobDG}) is 
much cheaper than (\ref{PbDG}) for large $k$. We have thus asymptotically the same costs for the global problems as for CG FEM with static condensation. There is though a fundamental difference between static condensation approaches in CG and DG cases from the implementation point of view: the basis functions for the global space $W_h'$ in the CG case are known \textit{a priori}, whereas those for the space $V_h'$ in the DG case should be constructed as solutions to local problems (\ref{piglobDGprob}), cf. the discussion of the implementation issues in Subsection \ref{secImpl}. Note however that one can get rid of problems (\ref{piglobDGprob}) in the special case of a constant coefficient matrix $A$ in (\ref{Pb}), cf. Remark \ref{RemConstA}. Indeed, (\ref{piglobDGprob}) is reduced in this case to $\LL(v'_h |_T)=A_{{ij}}\partial_i  \partial_j(v'_h |_T)=0$ on any cell $T\in\Th$. The structure of $V'_h$ does not thus vary from one cell to another and a basis for $V'_h$ can be chosen \textit{a priori} on all the cells.}

The local projection step (\ref{pilocDG}) is not necessarily consistent with the
original formulation (\ref{PbDG}) so that the solution $u_h$ given by
(\ref{pilocDG})--(\ref{sumDG}) is different from that of (\ref{PbDG}). We
shall prove however that SIP and scSIP approximations satisfy the same \textit{a priori} error
bounds. Moreover, they turn out to be very close to each other in numerical
experiments. 

\section{Well posedness of SIP and scSIP methods and
\textit{a priori} error estimates}
\label{sec3}

Let us now be more precise about the hypotheses on the mesh. Recall that  
$\Omega \subset \mathbb{R}^d$, $d = 2$ or 3, is a Lipschitz bounded domain
and $\mathcal{T}_h$ is a general (not necessarily polygonal or polyhedral)
mesh on $\Omega$. We mean by this that $\mathcal{T}_h$ is a decomposition of
$\Omega$ into mutually disjoint cells $\bar{\Omega} = \cup_{T \in
\mathcal{T}_h} \bar{T}$ so that each cell $T$ is a Lipschitz subdomain of
$\Omega$ and for every $T_1, T_2 \in \mathcal{T}_h$ we have either $T_1 = T_2$ 
or $T_1
\cap T_2 = \varnothing$ (the cells $T_1 \in \mathcal{T}_h$ are treated here as 
open sets). We also introduce the sets of internal and boundary edges/faces as 
respectively 
\begin{align*}
\Eh^i&=\{E=\bar{T_1}\cap\bar{T_2}\text{ for some }T_1,T_2\in\Th\} \\
\Eh^b&=\{E=\bar{T}\cap\partial\Omega\text{ for some }T\in\Th\}
\end{align*}
and denote by $\Eh:=\Eh^i\cap\Eh^b$ the union of all the edges/faces.

Let $B_T$, for any $T \in \mathcal{T}_h$, denote  the smallest ball containing 
$T$, 
and $B_T^{in}$ denote the largest ball inscribed in $T$. Set $h_T = \diam (T)$ 
and $h = \max_{T\in \mathcal{T}_h} h_T$. From now on, we assume that mesh $\Th$ 
is
\begin{itemize}
  \item \textbf{Shape regular}: there is a mesh-independent parameter $\rho_1 > 
1$ such
  that, for all $T \in \mathcal{T}_h$,
  \begin{equation}\label{shape} 
    R_T \leq \rho_1 r_T 
  \end{equation}
  where $r_T$ is the radius of $B_T^{in}$ and $R_T$ is the radius of $B_T$. This 
also implies $h_T \leq 2\rho_1 r_T$ and $R_T \leq \rho_1 h_T$.  
\end{itemize}
Choose an integer $k \ge 2$ and recall the discontinuous FE space (\ref{spacek}).
We assume that $V_h$ has two following properties (and we shall 
prove in Section \ref{SecMesh} that these properties hold under some 
additional assumptions on the mesh):
\begin{itemize}
 \item \textbf{Optimal interpolation}: there exists an operator 
$I_h:H^{k+1}(\Omega) \to V_h$ such that for any $v \in H^{k+1}(\Omega)$  
 \begin{align}\label{Interp}
  &\hspace{-5mm}
  \left( \sum_{T \in \mathcal{T}_h} \left( |v - I_hv |^2_{H^1 (T)}
  + \frac{1}{h^2_T} \|v - I_hv \|^2_{L^2 (T)} + h_T^2|v - I_hv |^2_{H^2 (T)}
  \right.\right. \\
  &  \left.\left. 
   + h_T \| \nabla v - \nabla I_hv \|^2_{L^2 (\partial T)} + \frac{1}{h_T}
  \|v - I_hv ||^2_{L^2 (\partial T)} \right)
  \right)^{\frac{1}{2}}
  \leq C \left( \sum_{T \in \mathcal{T}_h} h_T^{2k} |v|_{H^{k + 1} (T)}
  \right)^{\frac{1}{2}}   \notag
\end{align}
 \item \textbf{Inverse inequalities}: for any $v_h \in V_h$ and any $T\in\Th$
 \begin{equation}\label{traceinv}
 \|v_h \|_{L^2  (\partial T)} \leq \frac{C}{\sqrt{h_T}} \|v_h\|_{L^2(T)}  
  \qquad
  \| \nabla v_h \|_{L^2  (\partial T)} \leq \frac{C}{\sqrt{h_T}}  \| \nabla v_h 
\|_{L^2 (T)} 
 \end{equation}
 and
 \begin{equation}\label{H2inv}
  | v_h |_{H^2  (T)} \leq \frac{C}{{h_T}}  |v_h |_{H^1 (T)} 
 \end{equation}
\end{itemize}

We can now study the well posedness and establish optimal \textit{a priori} error 
estimates for the classical SIP method (\ref{PbDG}). 

\begin{lemma}\label{lemma1}
Under the above assumptions on the mesh and on $V_h$, setting 
\begin{align}\label{hE}
 h_E &= 2\left(\frac{1}{h_{T_1}} + \frac{1}{h_{T_2}}\right)^{-1}&&\text{ for any 
}E\in\Eh^i\text{ with }E = \partial T_1 \cap \partial T_2, \\
 h_E &= h_T&&\text{ for any }E\in\Eh^b\text{ with }E = \partial T \cap 
\partial\Omega, \notag
\end{align}
and choosing $\gamma$ large enough, $\gamma\ge\gamma_0$, the bilinear form $a_h$ defined by (\ref{ah})
is coercive, i.e.
\begin{equation}\label{ahcoer}
a_h(v_h , v_h ) \ge c \interleave v_h\interleave ^2,\quad \forall v_h\in V_h
\end{equation}
with some $c > 0$ and the triple norm defined by
\begin{equation}\label{triple}
\interleave v\interleave ^2 = \sum_{T\in\Th} 
  \left( |v|^2_{H^1(T)} + \frac{1}{h_T} \|[v]\|^2_{L^2(\partial T)} \right)
\end{equation}
The constants $c,\gamma_0$ depend only on the parameters in the assumptions on 
the mesh and on $V_h$, as well as on $\alpha,\beta$ in (\ref{regel1}).
\end{lemma}

\new{We skip the proof of this well known result. We stress however that our definitions of the length scale $h_E$ and of the triple norm (\ref{triple}) may be slightly different from those available in the literature. In particular, we avoid to use the diameter of a facet $E$ (or any other geometrical information on $E$) to define $h_E$. This choice enables us to establish straightforwardly the coercivity of $a_h$ with respect to the triple norm, which does not see the separate mesh facets either (only the whole boundaries of the mesh cells are present there). More elaborate choices for the interior penalty parameters are proposed in \cite{cangiani14}. }

Lemma \ref{lemma1} implies that problem (\ref{PbDG}) of the SIP method is well posed. Moreover, we 
have the following error estimate, the proof of which is also skipped (actually, it goes along the same 
lines as that of our forthcoming Theorem \ref{MainTh}).

\begin{theorem}
Assume that the solution u to (\ref{Pb}) is in $H^{k+1}(\Omega)$. Under the 
assumptions of Lemma \ref{lemma1}, there exists the unique solution $u_h$ to 
(\ref{PbDG}) and it satisfies
$$
|u - u_h |_{H^1(\Th)} \le Ch^k |u|_{H^{k+1}(\Omega)} 
$$
where $H^1(\Th)$ is the broken $H^1$ space on the mesh $\Th$ and 
$|\cdot|_{H^1(\Th)}:=\left(\sum_{T\in\Th} 
|\cdot|^2_{H^1(T)}\right)^{\frac{1}{2}}$.  
  If, moreover, the elliptic regularity property holds for (\ref{Pb}), then
  \[ \|u - u_h \|_{L^2 (\Omega)} \leq C |u|_{H^{k + 1} (\Omega)} h^{k +
     1} \]
\end{theorem}
 
We turn now to the study of the scSIP method 
(\ref{pilocDG})--(\ref{PbglobDG}) and start by the
following technical lemma.
\begin{lemma}\label{lemma2}
  There exists $h_0 > 0$ such that for all $T \in \mathcal{T}_h$ with $h_T
  \leq h_0$ and for all $q_T \in \mathbb{P}_{k - 2} (T)$ one can find
  $u_T \in \mathbb{P}_k (T)$ such that
  \begin{equation}
    \label{uT1} \int_T q_T ({\LL u}_T) \geq \frac{1}{2} \|q_T \|_{L^2
    (T)}^2
  \end{equation}
  and
  \begin{equation}
    \label{uT2} |u_T |_{H^1 (T)}^2 + \frac{1}{h_T} \|u_T \|_{L^2  (\partial
    T)}^2 \leq {Ch}_T^2 \|q_T \|_{L^2 (T)}^2
  \end{equation}
  The constants $h_0$ and $C$ depend only on the regularity of the mesh and
  on $\alpha$, $\beta$ and $M$ in (\ref{regel1}) and (\ref{regel2}). \new{One can put $h_0=+\infty$ if the coefficient matrix $A$ is constant on $T$.}
\end{lemma}

\begin{proof}
  Let $\chi_T$ be the polynomial of degree 2 vanishing on $\partial
  B_T^{{in}}$, i.e. $$\chi_T (x) = \left( \sum_{i = 1}^d (x_i - x^0_i)^2 -
  r_T^2 \right)$$ where $x^0 = (x^0_1, \ldots, x^0_d)$ is the center of
  $B_T^{{in}}$ and $r_T$ is its radius. Set $A^0_{{ij}} =
  A_{{ij}} (x^0)$ and $\LL^0$=$- \partial_i A_{{ij}}^0 \partial_j$.
  Consider the linear map
  \[ Q : \mathbb{P}_{k - 2} (T) \rightarrow \mathbb{P}_{k - 2} (T) \]
  defined by
  \[ Q (v) = \LL^0  (\chi_T v) \]
  The kernel of $Q$ is $\{0\}$. Indeed, if $Q (v) = 0$ then $w:= \chi_T v$ is 
the solution to
\[ \LL^0 w = 0 \text{ in } B_T^{{in}}, \hspace{1em} w = 0 \text{ on } \partial
   B_T^{{in}} \]
so that $w = 0$ as a solution to an elliptic problem with vanishing right-hand 
side and boundary conditions. 
Since $Q$ is a linear map on the finite dimensional
space $\mathbb{P}_{k - 2} (T)$, this means that $Q$ is
one-to-one.
  
  Take any $q_T \in \mathbb{P}_{k - 2} (T)$ and let $u_T = \chi_T v_T$ with
  $v_T \in \mathbb{P}_{k - 2} (T)$ such that $Q (v_T) = q_T$. We have thus
  constructed $u_T \in \mathbb{P}_k (T)$ such that $\LL^0 u_T = q_T$. This
  immediately proves (\ref{uT1}) in the case of an operator $\LL = \LL^0$ with
  constant coefficients. Moreover, by scaling,
  \begin{equation}
    \label{uT3} |u_T |_{W^{2, \infty} (B_T)} + \frac{1}{h_T} |u_T |_{W^{1,
    \infty} (B_T)} + \frac{1}{h_T^2} \|u_T \|_{L^{\infty} (B_T)} \leq
    \frac{C}{h_T^{d / 2}} \|q_T \|_{L^2 (B_T^{{in}})}
  \end{equation}
  with a constant $C$ depending only on $\alpha$, $\beta$ and the ratio $R_T /
  r_T$. Thus,
  \[ |u_T |_{H^1 (T)} \leq | T |^{1 / 2} | u_T |_{W^{1, \infty} (B_T)}
     \leq {Ch}_T \|q_T \|_{L^2 (T)} \]
  which proves the estimate in $H^1 (T)$ norm in (\ref{uT2}). Similarly, $\|
  u_T \|_{L^2 (T)} \leq {Ch}_T^2 \|q_T \|_{L^2
  (T)}$ and the estimate in $L^2 (\partial T)$ norm in
  (\ref{uT2}) follows by the trace inverse inequality.
  
  It remains to prove (\ref{uT1}) in the case of operator $\LL$ with variable
  coefficients. To this end, we use the estimates in (\ref{uT3}) as follows
  
  \begin{align*}
    \int_T q_T {\LL u}_T & = \int_T q_T \LL^0 u_T + \int_T q_T \partial_i 
    ((A_{{ij}} - A_{{ij}}^0) \partial_j u_T)\\
    & \geq \|q_T \|_{L^2 (T)}^2 - \|q_T \|_{L^2 (T)} |T|^{1
    / 2} [\max_{x \in T} |A (x) - A_0 ||u_T |_{W^{2, \infty} (T)} + \max_{x
    \in T} | \nabla A (x) ||u_T |_{W^{1, \infty} (T)}]\\
    & \geq \|q_T \|_{L^2 (T)}^2 - \|q_T \|_{L^2 (T)} |T|^{1 / 2} h_T
    \max_{x \in T} | \nabla A (x) | \frac{C}{h_T^{d / 2}} \|q_T \|_{L^2
    (B_T^{{in}})}\\
    & \geq \|q_T \|_{L^2 (T)}^2 - {Ch}_T \|q_T \|^2_{L^2 (T)}
    \geq \frac{1}{2} \|q_T \|_{L^2 (T)}
  \end{align*}  
  for sufficiently small $h_T$. 
\qed\end{proof}

\begin{corollary}\label{corol2}
Introduce the bilinear form 
\[ b_h (q, v) = \sum_{T \in \mathcal{T}_h} h_T^2  \int_T {q \LL u} \]
and the space
\[ M_h = \{v \in L^2 (\Omega) : v|_T \in \mathbb{P}_{k - 2} (T),
   \forall T \in \mathcal{T}_h \} \]
Equip the space $V_h$ with the triple norm (\ref{triple})  and the 
space $M_h$ with
  \[ \|q\|_h = \left( \sum_{T \in \mathcal{T}_h} h_T^2 \|q\|_{L^2 (T)}^2
     \right)^{1 / 2} \]
  The bilinear form $b_h$ satisfies the inf-sup condition
  \begin{equation}\label{infsupb}
    \inf_{q_h \in M_h} \sup_{v_h \in V_h}  \frac{b_h (q_h, v_h)}{\|q_h \|_h
     \interleave v_h \interleave} \geq \delta
  \end{equation}
 with a mesh-independent constant $\delta>0$. Moreover, $b_h$ is continuous on $M_h\times V_h$ with a mesh-independent continuity bound. 
\end{corollary}
\begin{proof}
  Take any $q_h \in M_h$, denote $q_T = q_h |_T$, construct
  $u_T$ as in Lemma \ref{lemma2} and introduce $u_h \in V_h$ by $u_h |_T =u_T$ on
  all $T \in \mathcal{T}_h$. This yields using (\ref{uT1}) and (\ref{uT2}),
  \[ \frac{b_h (q_h, u_h)}{\interleave u_h \interleave} \geq
     \frac{\sum_{T \in \mathcal{T}_h} \frac{h^2_T}{2} \|q_h \|^2_{L^2
     (T)}}{\left( \sum_{T \in \mathcal{T}_h} {Ch}^2_T \|q_h \|^2_{L^2
     (T)} \right)^{1 / 2}} = \frac{2}{\sqrt{C}} \|q_h \|_h \]
  which is equivalent to (\ref{infsupb}) with $\delta = 2 / \sqrt{C}$.
  Finally, the continuity of $b_h$ is easily seen from the inverse inequality (\ref{H2inv}). 
\qed\end{proof}

Lemma \ref{lemma2} implies that operator $\pi_{T,k-2}\mathcal{L}$ appearing in (\ref{pilocDG}) is 
surjective from $\mathbb{P}^k(T)$ to $\mathbb{P}^{k-2}(T)$
 so that (\ref{pilocDG}) has indeed a solution at least on sufficiently refined meshes. The 
existence of a solution to
(\ref{PbglobDG}) follows from the coercivity of $a_h$. Thus, scheme (\ref{pilocDG})--(\ref{sumDG}) produces some
$u_h\in V_h$. In order to establish the error estimates for this $u_h$, we reinterpret its 
definition as a saddle point problem.

\begin{lemma}\label{lemma3}
The problem of finding $u_h\in V_h$ and $p_h \in M_h$ such that
\begin{align}
 a_h (u_h, v_h) + b_h (p_h, v_h) &= L_h(v_h), &&
   \forall v_h \in V_h 
   \label{21}\\
 b_h (q_h, u_h) &= \sum_{T \in \mathcal{T}_h} h_T^2  \int_{T}fq_h, &&
   \forall q_h \in M_h 
   \label{22}
\end{align}
has a unique solution. Moreover, $u_h$ given by (\ref{21})--(\ref{22}) coincides with $u_h$ given by (\ref{pilocDG})--(\ref{sumDG}).\footnote{\new{More precisely, all solutions $u_h$ of (\ref{pilocDG})--(\ref{sumDG}) may be accompanied by $p_h\in M_h$ so that the resulting couples $(u_h,p_h)$ also solve (\ref{21})--(\ref{22}). Since the solution to (\ref{21})--(\ref{22}) is unique, the inverse statement is also true: $u_h$ given by (\ref{21})--(\ref{22}) is also a solution to (\ref{pilocDG})--(\ref{sumDG}).}} This implies that $u_h$ produced by the scheme (\ref{pilocDG})--(\ref{sumDG}) is unique. 
\end{lemma}
\begin{proof} 
The existence and uniqueness of the solution to (\ref{21})--(\ref{22}) follows from the standard theory of saddle point problems, cf. for example Corollary 4.1 from \cite{girault}, thanks to the coercivity of $a_h$ (Lemma \ref{lemma1}) and to the inf-sup property on $b_h$ (Corollary \ref{corol2}). 
 
 In order to explore its relation with $u_h=u^{loc}_h+u'_h$ from (\ref{pilocDG})--(\ref{sumDG}), we note $b_h(q_h, u^{loc}_h )= \sum_{T\in\Th} h_T^2 \int_T f q_h$ for all $q_h \in M_h$ and $b_h(q_h, u_h') = 0$ for all $q_h \in M_h$ since
$u_h' \in V_h'$. We obtain thus 
\begin{equation}\label{22alt}
b_h (q_h, u_h^{loc} + u_h') = \sum_{T \in \mathcal{T}_h} h_T^2  \int_{T}{fq}_h, \quad
   \forall q_h \in M_h 
\end{equation}
Eq. (\ref{PbglobDG}) can be rewritten as
$$ L_h(v_h')-a_h (u_h^{loc} + u_h', v_h')= 0, \quad \forall v_h' \in V_h'$$ 
\new{This, together with the fact that $V_h'$ is precisely the kernel of the bilinear form $b_h$, i.e. $V_h'=\{v_h\in V_h:b(q_h,v_h)=0,\ \forall q_h\in M_h\}$, means that there exists $\tilde{p}_h\in M_h$ such that
$$
b_h(\tilde{p}_h,v_h)=L_h(v_h)-a_h(u_h^{loc} + u_h', v_h), \quad \forall v_h \in V_h,
$$
cf. Lemma 4.1 from \cite{girault}. The last equation can be rewritten as
\begin{equation}\label{21alt}
a_h(u_h^{loc} + u_h', v_h)+b_h(\tilde{p}_h,v_h)=L_h(v_h), \quad \forall v_h \in V_h,
\end{equation}
Comparing (\ref{21alt})--(\ref{22alt}) on one hand with (\ref{21})--(\ref{22}) on the other hand, we identify $u_h$ with $u_h^{loc} + u_h'$ and $p_h$ with $\tilde{p}_h$.}
\qed\end{proof}

\begin{theorem}\label{MainTh}
  Assume that the solution $u$ to (\ref{Pb}) is in $H^{k + 1} (\Omega)$. Under 
the assumptions of Lemma \ref{lemma1} and $h$ sufficiently small, the scSIP method (\ref{pilocDG})--(\ref{sumDG}) produces 
the unique solution $u_h\in V_h$, which  satisfies
  \begin{equation}\label{H1err}
    |u - u_h |_{H^1 (\mathcal{T}_h)} \leq C h^k |u|_{H^{k + 1} (\Omega)} 
  \end{equation}
  If, moreover, the elliptic regularity property holds for (\ref{Pb}), then
  \begin{equation}\label{L2err}
   \|u - u_h \|_{L^2 (\Omega)} \leq C |u|_{H^{k + 1} (\Omega)} h^{k + 1} 
  \end{equation}
\end{theorem}

\begin{proof}
  We shall use the saddle point reformulation (\ref{21})--(\ref{22}). 
This discretization is consistent. Indeed setting $p = 0$ we have
  \begin{align*}
   a_h (u, v_h) + b_h (p, v_h) &= L_h(v_h)
     , &&\forall v_h \in V_h \\
   b_h (q_h, u) &= \sum_{T \in \mathcal{T}_h} h_T^2  \int_T
     {fq}_h, &&\forall q_h \in M_h 
  \end{align*}
  Thus, by the standard approximation theory for saddle point problems, cf. for example 
Proposition 2.36 from \cite{ern}, recalling the coercivity of $a_h$ (Lemma \ref{lemma1}) and the inf-sup property on $b_h$ (Corollary \ref{corol2}), we get   
  \begin{align*}
    \interleave u_h - I_h u \interleave + \|p_h \|_h &\leq
   C\sup_{\scriptsize{\begin{array}{l}
     (v_h, q_h) \in V_h \times M_h
     \notag\\
     \interleave v_h \interleave + \| q_h \|_h = 1
   \end{array}}} (a_h (u_h - I_h u, v_h) + b_h (p_h, v_h) + b_h (q_h, u_h -
   I_h u)) 
   \notag\\
   &= C\sup_{\scriptsize{\begin{array}{l}
     (v_h, q_h) \in V_h \times M_h\\
     \interleave v_h \interleave + \| q_h \|_h = 1
   \end{array}}} (a_h (u_{} - I_h u, v_h) + b_h (q_h, u_{} - I_h u)) \\
  &\leq C \left( \interleave u - I_h u \interleave_a^2 + \sum_{T \in
   \mathcal{T}_h} h_T^2 \| \LL (u - I_h u) \|_{L^2 (T)}^2
   \right)^{\frac{1}{2}} 
  \end{align*}
with the augmented triple norm $\interleave \cdot \interleave_a$ defined by
$$
\interleave v \interleave_a^2 := \interleave v \interleave^2 +
  \sum_{E \in \mathcal{E}_h} h_E\| \{A \nabla v \cdot n\} \|_{L^2(E)}^2
$$
Applying the interpolation estimates (\ref{Interp}) and the 
triangle inequality gives 
\begin{equation}
  \interleave u - I_h u \interleave_a + \|p_h \|_h \leq C h^k |u|_{H^{k + 1}}
      \label{estuhph} 
\end{equation}
This implies in particular (\ref{H1err}).

To prove the $L^2$ error estimate, we consider the auxiliary problem for $z\in H^2(\Omega)$
  \[ {\LL z} = u - u_h \text{ in } \Omega , \quad z = 0 \text{ on } \partial\Omega 
\]
  Then, for all $v \in H^1 (\mathcal{T}_h)$ and $q \in L^2 (\Omega)$,
  \[ a_h (v, z) + b_h (q, z) = \int_{\Omega} (u - u_h) v + \sum_{T
     \in \mathcal{T}_h} h_T^2  \int_T (u - u_h) q \]
  Setting $v = u
  - u_h$ and $q = p - p_h$ (with $p=0$) and using Galerkin orthogonality yields
    \begin{align*}
    \|u - u_h \|_{L^2 (\Omega)}^2 + \sum_{T \in \mathcal{T}_h} h_T^2  \int_T
    (u - u_h)  (p - p_h) & = a_h  (u - u_h, z) + b_h  (p - p_h, z)\\
    & = a_h  (u - u_h, z - z_h) + b_h  (p - p_h, z - z_h)
  \end{align*}
    for any $z_h \in V_h$. Thus, taking $z_h=I_h z$ and applying the interpolation estimates,
  \begin{align*}
    \|u - u_h \|_{L^2 (\Omega)}^2 & \leq |a_h (u - u_h, z -
    z_h) | + h \|p - p_h \|_h \left( \sum_{T \in \mathcal{T}_h} \|\LL z - \LL z_h
    \|_{L^2 (T)}^2 \right)^{1 / 2} \\
    &\qquad + h \|p - p_h \|_h  \|u - u_h\|_{L^2 (\Omega)}\\
    & \leq {Ch} \interleave u-I_hu \interleave_a |z|_{H^2 (\Omega)}
    + {Ch} \|p_h \|_h  (|z|_{H^2 (\Omega)} +\|u - u_h
    \|_{L^2 (\Omega)})
  \end{align*}
   Recalling $|z|_{H^2 (\Omega)} \leq C \|u - u_h \|_{L^2(\Omega)}$ and (\ref{estuhph}) yields (\ref{L2err}).
\qed\end{proof}

\section{An example of assumptions on the mesh that guarantee the
interpolation and inverse estimates}\label{SecMesh}

In this section, we adopt the following assumptions on the mesh.
\begin{description}
  \item[M1:] $\Th$ is shape regular  in the sense (\ref{shape}) with a parameter 
$\rho_1 >  1$.
  
  \item[M2:] $\Th$ is locally quasi-uniform  in the following sense: \new{for any two mesh cells $T,T' \in \mathcal{T}_h$ such that $B_{T'}\cap B_T\not=\varnothing$ there holds
  \[ \frac{1}{\rho_2} h_{T'} \leq h_T \leq \rho_2 h_{T'} \]
  with a parameter $\rho_2 >  1$.}
  \item[M3:] The cell boundaries are not too wiggly: for all $T \in 
\mathcal{T}_h$
  \[ | \partial T| \leq \rho_3 h_T^{d - 1} \]
  with a parameter $\rho_3 >  0$.
\end{description}
We shall show that these assumptions allow us to construct an interpolation operator $I_h$ to the discontinuous finite element space (\ref{spacek}) for $k\ge 2$ and to prove the interpolation error estimate (\ref{Interp}) and the inverse estimates (\ref{traceinv})--(\ref{H2inv}).

{\medskip}First of all, the assumptions that the mesh is shape regular and 
locally quasi-uniform entail the following\\
\begin{lemma}\label{nballs} \new{Define, for any $x\in\mathbb{R}^d$, 
$$N_{{ball}} (x)=\#\{ T\in\Th : x\in B_T\}$$
with $\#$ standing for the ``number of''. 
Under assumptions M1 and M2, there holds $$N_{{ball}} (x)\le N_{{int}},\quad \forall 
x\in\mathbb{R}^d$$} with a constant $N_{{int}}$ depending only on $\rho_1$ and 
$\rho_2$.
\end{lemma}
\begin{proof}
\new{Take any $x\in\mathbb{R}^d$ with $N_{{ball}} (x)>0$ (otherwise, for $x$ with $N_{{ball}} (x)=0$, there is nothing to prove). We now choose arbitrarily $T' \in \mathcal{T}_h$ such that $x \in B_{T'}$, set $h_x=h_{T'}$, and then consider all the mesh cells $T$ such that $x \in B_{T'}$.} By Assumption M2, $h_T \leq \rho_2 h_x$ for any such $T$. Hence, by Assumption M1, $R_T \leq \rho_1\rho_2 h_x$, so that $T$ is inside the ball $B_x$ of radius $2\rho_1 \rho_2 h_x$ centered
at $x$. Recall that $T$ contains an inscribed ball of radius
$r_T \geq \frac{R_T}{\rho_1} \geq \frac{h_T}{2\rho_1} \geq 
\frac{h_x}{2\rho_1\rho_2}$. If there are
several such cells $T$, then their respective inscribed balls $B_T^{in}$ do not 
intersect each other and they are all inside $B_x$. Thus, their number satisfies the bound 
\[ N_{{ball}} (x) \le
  \frac{|B_x|}{\min\limits_{T \in \mathcal{T}_h : x\in B_T}|B_T^{in}|} \le
   \frac{(2\rho_1\rho_2 h_x)^d}{\left( \frac{h_x}{2\rho_1\rho_2}  \right)^d} = \left(
   {2\rho_1\rho_2} \right)^{2d} \]
as announced.\qed\end{proof}

Recall that $V_h$ is the discontinuous FE space on $\mathcal{T}_h$ of
degree $k \ge 2$, cf. (\ref{spacek}). 

\begin{lemma}\label{localInt} \textbf{(Local interpolation estimate)} Take any 
$T \in \mathcal{T}_h$. Let
$\pi_h$ denote the $L^2 (B_T)$-orthogonal projection to the space of
polynomials, i.e. given $v\in L^2(B_T)$, $v_h = \pi_h v$ is a polynomial of degree  
$\le k$ such that
\[ \int_{B_T} v_h \varphi_h = \int_{B_T} v \varphi_h \hspace{1em} \forall
   \varphi_h \in \mathbb{P}^k(T) \]
Under Assumptions M1 and M3, we have then for any $v \in H^{k + 1} (B_T)$
\begin{multline*} |v - v_h |_{H^1 (T)} + \frac{1}{h_T}  \|v - v_h
   \|_{L^2 (T)} + h_T |v - v_h |_{H^2 (T)} \\ + \sqrt{h_T}  \| \nabla( v - v_h) 
\|_{L^2  (\partial T)} + \frac{1}{\sqrt{h_T}} 
   \|v - v_h \|_{L^2  (\partial T)} \leq Ch_T^k |v|_{H^{k + 1}
   (B_T)} 
\end{multline*}   
with a constant $C>0$ depending only on $\rho_1$ and $\rho_3$. 
\end{lemma}
\begin{proof}
Since $H^{k + 1} (B_T)$ is embedded   into $L^{\infty} (B_T)$, 
\new{Deny-Lions lemma together with a scaling argument (cf. Theorem 15.3 from \cite{ciarlet91})} entail
\[ \|v - v_h \|_{L^{\infty} (B_T)} \leq Ch_T^{k + 1 - d / 2}
   |v|_{H^{k + 1} (B_T)} \]
Hence, 
\[ \|v - v_h \|_{L^2(T)} \leq |T|^{\frac 12}\|v - v_h \|_{L^{\infty} (B_T)} 
    \leq Ch_T^{k + 1}  |v|_{H^{k + 1} (B_T)} \]
and, in view of the hypothesis $| \partial T| \leq \rho_3 h_T^{d-1}$, 
$$\|v - v_h \|_{L^2  (\partial T)} \leq (\rho_3 h_T^{d-1})^{\frac 12}\|v - v_h 
\|_{L^{\infty} (B_T)} 
\leq Ch_T^{k + 1 / 2}|v|_{H^{k + 1} (B_T)}$$

 The estimates for $|v-v_h|_{H^1  (T)}$ and $\| \nabla v -\nabla v_h \|_{L^2  
(\partial T)}$ are proven in the same way starting from 
\[ \|\nabla v -\nabla v_h \|_{L^{\infty} (B_T)} \leq Ch_T^{k  - d / 2}   
|v|_{H^{k + 1} (B_T)} \]
This is valid since  $H^{k + 1} (B_T)$ is embedded into $W^{1,\infty} 
(B_T)$ for $k\geq 2$.

Finally, the estimate for $|v-v_h|_{H^2(T)}$ holds thanks to  the 
embedding   of $H^{k + 1} (B_T)$ into $H^{2} (B_T)$ ($k\geq 1$).

\qed\end{proof}

\begin{lemma}\label{globalInt}{\textbf{(Global interpolation estimate)}} Let 
$I_h : H^{k+1}
(\Omega) \rightarrow V_h$ denote the operator obtained by first extending a 
function $v\in H^{k+1}
(\Omega)$ by a function $\tilde{v}\in H^{k+1}(\mathbb{R}^d)$ and then applying 
the local
operator $\pi_h$ from Lemma \ref{localInt} to $\tilde{v}$ on every $T \in 
\mathcal{T}_h$, i.e. $I_hv|_T:=(\pi_h\tilde{v})|_T$ on any $T\in\Th$. 
Then, under Assumptions M1--M3, (\ref{Interp}) holds for any $v\in 
H^{k+1}(\Omega)$
\end{lemma}
\begin{proof}
First, extension theorem for Sobolev spaces \cite{Adams} insure that there 
exists $\tilde{v} \in H^{k + 1} (\mathbb{R}^d)$ such that
$$
\tilde{v}=v\text{ on }\Omega\quad\text{and}\quad \|\tilde{v}\|_{H^{k + 1} 
(\mathbb{R}^d)}\le C\|\tilde{v}\|_{H^{k + 1} (\Omega})
$$
To prove (\ref{Interp}), we sum the local interpolation estimates of Lemma 
\ref{localInt} over all the mesh
cells and then use Assumption M2 and Lemma \ref{nballs}: 
\begin{multline*}
  \sum_{T \in \mathcal{T}_h} \left( |v - v_h |^2_{H^1 (T)} +
  \frac{1}{h^2_T} \|v - v_h \|^2_{L^2 (T)}
	+ h_T^2|v - v_h |^2_{H^2(T)}
	+ h_T \| \nabla v - \nabla v_h \|^2_{L^2 (\partial T)} + \frac{1}{h_T}
  \|v - v_h ||^2_{L^2 (\partial T)} \right)\\
  \leq C \sum_{T \in \mathcal{T}_h} h_T^{2 k}\int_{B_T} |
  \nabla^{k + 1} v|^2 {dx} 
	\le C  \int_{_{\Omega}} \new{\left(\max\limits_{T \in \mathcal{T}_h : x\in B_{T}} h_{T}\right)^{2 k}} N_{{ball}} (x)  | \nabla^{k + 1} v|^2 {dx} \\ 
    \leq {CN}_{{int}} \new{\rho_2^{2k}} \sum_{T'\in \mathcal{T}_h} h_{T'}^{2 k} 
|v|^2_{H^{k + 1} (T)}
\end{multline*}
\qed\end{proof}

\begin{lemma}\label{LemInv}{\textbf{(Inverse inequalities)}} Under 
assumptions M1 and M3, (\ref{traceinv}) and (\ref{H2inv}) hold for any $v_h 
\in V_h$ and any $T \in \mathcal{T}_h$.
\end{lemma}
\begin{proof}
Both bounds in (\ref{traceinv}) follow immediately from the following one: for 
any polynomial
$q_h$ of degree $\le l$ one has
\[ \|q_h \|_{L^{\infty} (T)} \leq \frac{C}{h^{d / 2}_T} \|q_h\|_{L^2(T)} \]
with $C>0$ depending only on $\rho_1$ and $l$. This follows in turn from
\[ \|q_h \|_{L^{\infty} (B_T)} \leq \frac{C}{h^{d / 2}_T} 
\|q_h\|_{L^2(B_T^{{in}})} \]
where $B_T^{{in}}$ is the largest ball inscribed in $T$. Scaling the ball
$B_T$ to a ball of radius 1 $B_1$ and considering all the possible positions
of the inscribed ball, the last inequality can be rewritten as
\[ \|q_h \|_{L^{\infty} (B_1)} \leq C \min_{B^{{in}} \subset B_1
   , B^{{in}} \text{ a ball of radius } \geq \rho_1^{-1}} \|q_h
   \|_{L^2 (B_T^{{in}})} \]
This is valid for any polynomial of degree $l$ by equivalence of norms.

The remaining inverse inequality (\ref{H2inv}) can be proven similarly:
\[ | v_h |_{H^2 (T)} 
   \leq Ch^{d / 2 }_T |v_h |_{W^{2,\infty} (B_T)}
   \leq Ch^{d / 2 -1}_T |v_h |_{W^{1,\infty} (B_T)}
   \leq \frac{C}{h_T} |v_h|_{H^1 (B_T^{{in}})} 
   \leq\frac{C}{h_T} |v_h|_{H^1 (T)} \]
\qed\end{proof}

\section{Implementation and numerical results}
\label{sec5}

We shall illustrate the convergence of SIP and scSIP methods on polygonal meshes obtained by agglomerating the 
cells of a background triangular mesh. Both the mesh construction and the following calculations are done in FreeFEM++ \cite{freefem}. An example of such a mesh is given in Fig. \ref{FigMesh}. To construct it, we take a positive integer $n$ ($n=4$ in the Figure), let FreeFEM++ to construct a Delaunay triangulation of $\Omega=(0,1)^2$ with $4n$ boundary nodes on each side of the square, and finally agglomerate the triangles of this mesh into $n\times n$ cells as follows. We start by attributing the triangle containing the point 
\begin{equation}\label{Ocenter}
O_{i+jn}=\left(\frac{i-1/2}{n}, \frac{j-1/2}{n}\right),\quad i,j=1,\ldots,n 
\end{equation}
to the cell number $i+jn$. Then, iteratively, we run over all the cells and attach yet unattributed triangles neighboring a triangle from a cell to the same cell, until all the triangles are attributed. 

\begin{figure}[tbp]
\centerline{
\includegraphics[width=0.6\textwidth]{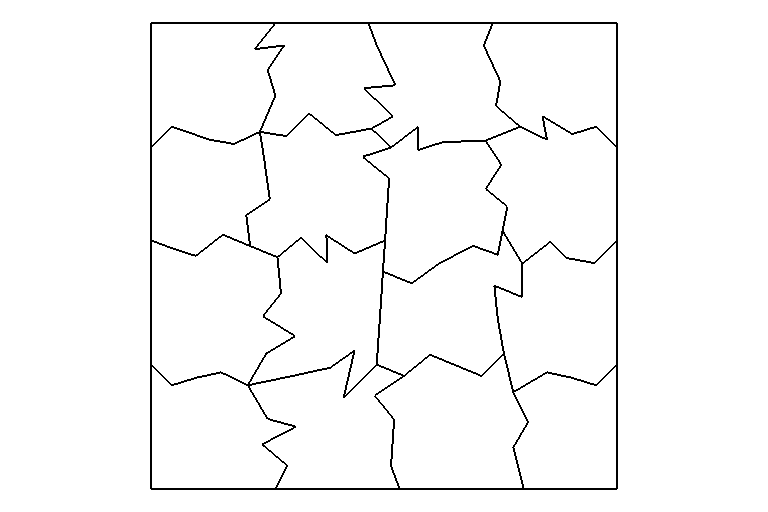}
\hspace{-12mm}
\includegraphics[width=0.6\textwidth]{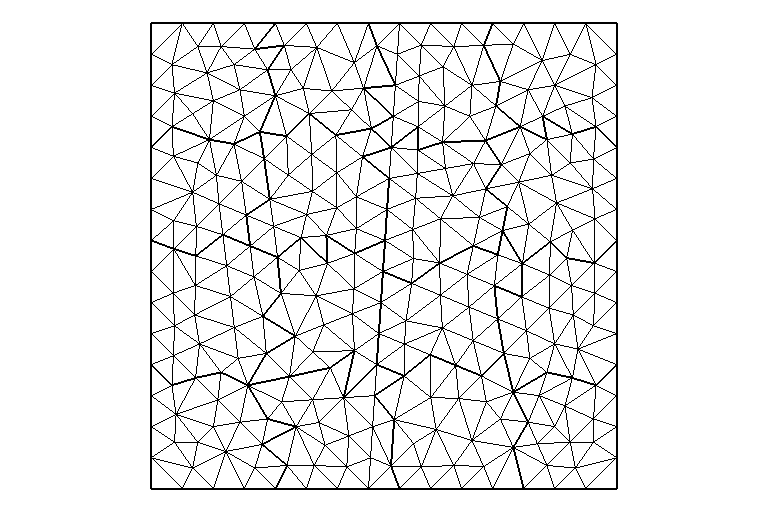}
}
\caption{On the left, a polygonal mesh consisting of $4 \times 4$ cells, which are obtained by the agglomeration of the triangles of a finer mesh seen on the right.
}
\label{FigMesh}
\end{figure}

Some details of our implementations are given below, followed by the numerical results on two test cases. 

\subsection{Implementation of SIP ans scSIP methods}\label{secImpl} 
\new{Let us enumerate the mesh cells as $\{ T_1, \ldots, T_{N_e}
\}$ and introduce a basis $\{ \phi^{(l)}_i \}_{i = 1, \ldots, N_k}$ of
$\mathbb{P}_k (T_l)$ on every cell $T_l$. Here and below, $N_e$ is the number of cells in $\Th$ and $N_k$ denotes the dimension of $\mathbb{P}_k$. In our implementation, we form the basis out of monomials shifted to the ``center'' $O_l=(O_{l,x},O_{l,y})$ of the cell $T_l$, cf. (\ref{Ocenter}). i.e.
$$
\phi^{(l)}_i(x,y)=\phi^{(l)}_{i_1i_2}(x,y)=(x-O_{l,x})^{i_1}(y-O_{l,y})^{i_2}
$$
regrouping the multi-indexes $(i_1,i_2)$, $0\le i_1\le k$, $0\le i_2\le k-i_1$ into a single index $i$ ranging from $1$ to $N_k$.  
We form then the matrices $\mathbf{A}^{(lm)}$ for every pair of cells $T_l$ and $T_m$ sharing some parts of their boundaries. These matrices of size $N_k \times N_k$ represent the bilinear form $a_h$ in our bases and have the following entries
\[ A_{ij}^{(lm)} = a_h (\phi_i^{(l)}, \phi^{(m)}_j) \]
We also compute the right-hand side vectors $\ovec{F}^{(l)}\in\RR^{N_k}$ with $F_i^{(l)} = L_h (\phi_i^{(l)})$ on every cell $T_l$, put all $\ovec{F}^{(l)}$ into a single vector $\ovec{F}$ of size $N_{DOF}=N_eN_k$, put the matrices $\mathbf{A}^{(lm)}$ into the block matrix $\mathbf{A}$ of size $N_{DOF}\times N_{DOF}$, and finally find $\ovec{U}\in\RR^{N_{DOF}}$ as solution to
\[ \mathbf{A} \ovec{U} = \ovec{F} \]
The vector $\ovec{U}$ represents the numerical solution by the SIP method (\ref{PbDG}) in the following sense: decomposing $\ovec{U}$ into the cell-by-cell components $\ovec{U}^{(l)}=\{{U}^{(l)}_i\}\in\RR^{N_k}$,  $u_h$ in (\ref{PbDG}) is given on each cell $T_l$ by $u_h=\sum_{i=1}^{N_k}{U}^{(l)}_i\phi^{(l)}_i$.
}

\new{Turning to the scSIP method, we introduce moreover a basis of $\mathbb{P}_{k - 2} (T_l)$, $\{ \psi^{(l)}_i \}_{i = 1, \ldots, N_{k - 2}}$, and form the matrices $\mathbf{B}^{(l)}$ of size $N_{k-2} \times N_{k}$ on every cell $T_l$ with the entries
\[ B^{(l)}_{ij} = \int_{T_l} \psi^{(l)}_i \mathcal{L} \phi^{(l)}_j = \int_{T_l}
   \psi^{(l)}_i \cdot A \nabla \phi^{(l)}_j - \int_{\partial {T_l}} \psi^{(l)}_i n \cdot A
   \nabla \phi^{(l)}_j \]
These matrices will serve to compute the local contributions in (\ref{pilocDG})  as well as to construct a basis of the space $V_h'$ in (\ref{piglobDG}). As mentioned earlier, the solution to (\ref{pilocDG}) is not unique and one can propose several ways to compute a solution in practice. In our implementation, we have opted for a solution to (\ref{pilocDG}) solving the following saddle-point problem on every cell $T_l$
\begin{align}
\ovec{u}^{(l)} + (\mathbf{B}^{(l)})^T \ovec{p}^{(l)} &= 0 
\notag\\
\mathbf{B}^{(l)} \ovec{u}^{(l)} &= \ovec{F_{\psi}}^{(l)} 
\label{saddle1}
\end{align}
with $\ovec{F_{\psi}}^{(l)}=\{F^{(l)}_{\psi, i}\}\in\RR^{N_{k-2}}$, $F^{(l)}_{\psi, i} = \int_{T_l} f \psi^{(l)}_i$. The unknowns here are $\ovec{u}^{(l)}\in\RR^{N_k}$ and $\ovec{p}^{(l)}\in\RR^{N_{k-2}}$ with $\ovec{u}^{(l)}$ representing $u_h^{loc}$ on $T_l$ in the basis $\{\phi_i^{(l)}\}$. The saddle-point problem above is well posed thanks to Lemma \ref{lemma2}.
}

\new{A basis for $V_h'$ from (\ref{piglobDG})--(\ref{piglobDGprob}) can be constructed on every cell $T_l$ using a saddle-point problem similar to (\ref{saddle1}). Indeed, $V_h'$ on $T_l$ is the kernel of $\mathbf{B}^{(l)}$. It is thus given by the span of vectors $\{\ovec{u}^{(l, 1)},\ldots,\ovec{u}^{(l, N_k)}\}\subset\RR^{N_{k}}$ with $\ovec{u}^{(l, s)}$ defined by
\begin{align}
 \ovec{u}^{(l, s)} + (\mathbf{B}^{(l)})^T \ovec{p}^{(l, s)} &= \ovec{e}^{(s)} 
  \notag\\
 \mathbf{B}^{(l)} \ovec{u}^{(l, s)} &= 0 
\label{saddle2}
\end{align}
where $\{ \ovec{e}^{(1)}, \ldots, \ovec{e}^{(N_k)} \}$ is the canonical basis of
$\mathbb{R}^{N_k}$. In practice, we solve the problem above successively for $s=1,2,\ldots$ and apply the Gram-Schmidt procedure to ortho-normalize the vectors  $\ovec{u}^{(l, s)}$ getting rid of the vectors which turn out to be linearly dependent from the preceding ones. This provides us with a basis for $V_h'|_{T_l}$ consisting of $N_k'=N_k-N_{k-2}$ vectors (actually, the Gram-Schmidt process can be stopped once $N_k'$ ortho-normal vectors have been found). 
}

\begin{remark}\label{RemConstA} 
\new{In the case when the coefficient matrix $A$ is constant (and thus
does not change from one mesh cell to another), the restriction on the
functions in $V_h'$, i.e. $\int_T q_h \mathcal{L}v_h' = 0$ for all $q_h \in
\mathbb{P}_{k - 2} (T)$, implies in fact $\mathcal{L}v_h' = 0$ on every cell
$T$. This is independent from the shape of $T$ so that the structure of $V_h'$ is the same on all the cells, and one can keep the same basis for $V_h'$ everywhere.
}

\new{For example, in the case of Poisson equation ($\mathcal{L}=-\Delta$) in 2D, $v_h$ supported on a cell $T$ is in $V_h'$ if and only if $\Delta v_h=0$. Expanding $v_h$ in the basis of monomials $v_h=\sum_{i_1,i_2}v_{i_1i_2}\phi^{(l)}_{i_1i_2}$ this gives rise to the equations 
$$(i_1+2)(i_1+1)v_{i_1+2,i_2}+(i_2+2)(i_2+1)v_{i_1,i_2+2}=0$$
for all non-negative $(i_1,i_2)$. These equations can be easily solved to provide a basis for $V_h'$ on all the cells.
}

\new{In our implementation, to keep things simple and the code suitable for both cases of either constant $A$ or varying $A$, we have used another strategy: if $A$ is constant, we perform the  Gram-Schmidt ortho-normalization on the solutions to (\ref{saddle2}) on the mesh cell number 1 only. We keep then the same basis (as expressed by the expansion coefficients in $\{\phi_i^{(l)}\}_{i=1,\ldots,N_k}$) on all the other cells $T_l$, $l\ge 2$. }
\end{remark}

\new{Having constructed the basis for $V_h'$, it remains to  solve the global problem (\ref{PbglobDG}). We introduce to this end on every cell $T_l$ the matrices $\mathbf{M}^{(l)}$ of size $N_k \times N_k'$ putting together the vectors representing the basis for $V_h'$ on $T_l$. We form then the reduced matrices $\mathbf{A}^{\prime (lm)}$ of size $N_k' \times N_k'$ and the reduced right-hand side vectors $\ovec{F}^{\prime (l)}$ out of full $\mathbf{A}^{(lm)}$ and $\ovec{F}^{(l)}$ (already introduced in the description of the SIP method implementation) by, cf. (\ref{PbglobDG}),
\[ \mathbf{A}^{\prime (lm)} = (\mathbf{M}^{(l)})^T \mathbf{A}^{(lm)}
   \mathbf{M}^{(m)} \text{ and } \ovec{F}^{\prime (l)} = (\mathbf{M}^{(l)})^T \left(  \ovec{F}^{(l)} - \sum_m \mathbf{A}^{(lm)}  \ovec{u}^{(m)} \right)
\]
with $\ovec{u}^{(m)}$ representing $u_h^{loc}$ on $T_m$ and computed by (\ref{saddle1}).
Putting the matrices $\mathbf{A}^{\prime (lm)}$ into the block matrix $\mathbf{A}'$ of size $N_{DOF}'\times N_{DOF}'$ with $N_{DOF}'=N_eN_k'$ and the vectors $\ovec{F}^{\prime (l)}$ into a single vector $\ovec{F}'$, we compute $\ovec{U}'\in\RR^{N_{DOF}'}$ as solution to
\[ \mathbf{A}' \ovec{U}' = \ovec{F}' \]
The vector $\ovec{U}'$ represents the solution to (\ref{PbglobDG}). The solution $u_h$ by the scSIP method is finally reconstructed as follows: decomposing $\ovec{U}'$ into the cell-by-cell components $\ovec{U}^{\prime (l)}=\{{U}^{\prime (l)}_i\}\in\RR^{N_k'}$, we recall $\ovec{u}^{(l)}$ computed on each cell $T_l$ by (\ref{saddle1}), introduce $\ovec{\widetilde U}^{(l)}=\{{\widetilde U}^{(l)}_i\}\in\RR^{N_k}$ as $\ovec{\widetilde U}^{(l)}=\ovec{u}^{(l)}+\mathbf{M}^{(l)}\ovec{U}^{\prime (l)}$, and set $u_h$ on $T_l$ as $u_h=\sum_{i=1}^{N_k}{\widetilde U}^{(l)}_i\phi^{(l)}_i$.
}

\subsection{The first test case: Poisson equation}

\begin{figure}[tbp]
\centerline{
\includegraphics[width=0.48\textwidth]{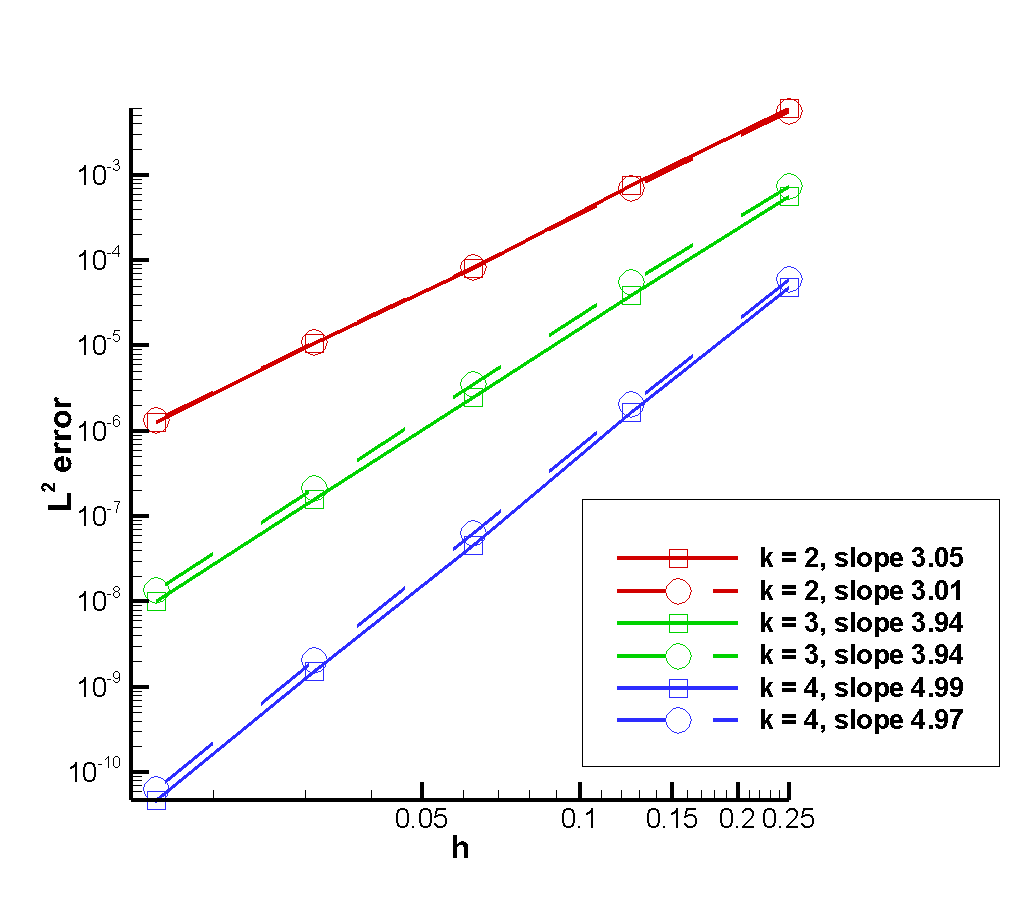}
\quad
\includegraphics[width=0.48\textwidth]{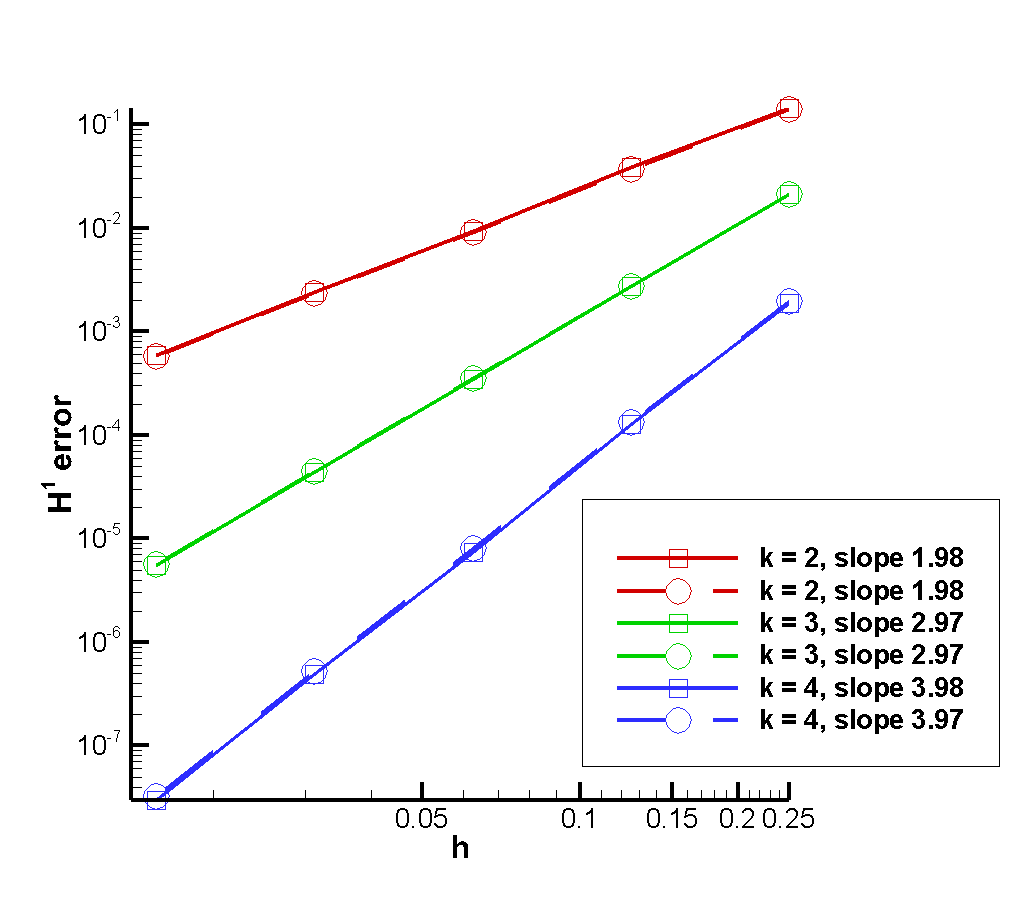}
}
\caption{The test case with Poisson equation: the error in $L^2$ norm and $H^1$ semi-norm vs. mesh-size $h$. The solid lines with squares represent the SIP method. The dashed lines with circles represent the scSIP method. }
\label{FigConv}
\end{figure}

We have considered the Poisson equation, i.e. (\ref{Pb}) with $A=I$, on $\Omega=(0,1)^2$ with homogeneous Dirichlet boundary conditions $g=0$ and the exact solution $u=\sin(\pi x)\sin(\pi y)$. We have applied SIP method (\ref{PbDG}) and scSIP method (\ref{pilocDG})--(\ref{PbglobDG}) to this problem on the agglomerated meshes as described in the preamble of this Section. The results are presented in Fig. \ref{FigConv}. In a slight deviation from the general notations, we set here the mesh-size as $h=1/n$ on the $n\times n$ mesh, and $h_E=h$ on all the edges in (\ref{ah}). Three choices for the polynomial space degree $k$ were investigated, namely $k=2,3,4$ and the penalty parameter $\gamma$ in (\ref{ah}) was set to $2k(k+1)$ (by a loose extrapolation to the polygonal meshes of the bound on the constant in the inverse inequality on a triangle in \cite{warburton03}). The numerical results confirm the theoretically expected order of convergence in both $L^2$ norm and $H^1$ semi-norm. They also demonstrate that the approximation produced by SIP and scSIP methods are very close to each other.

\subsection{The second test case: non-constant coefficients $A$}

\begin{figure}[tbp]
\centerline{
\includegraphics[width=0.48\textwidth]{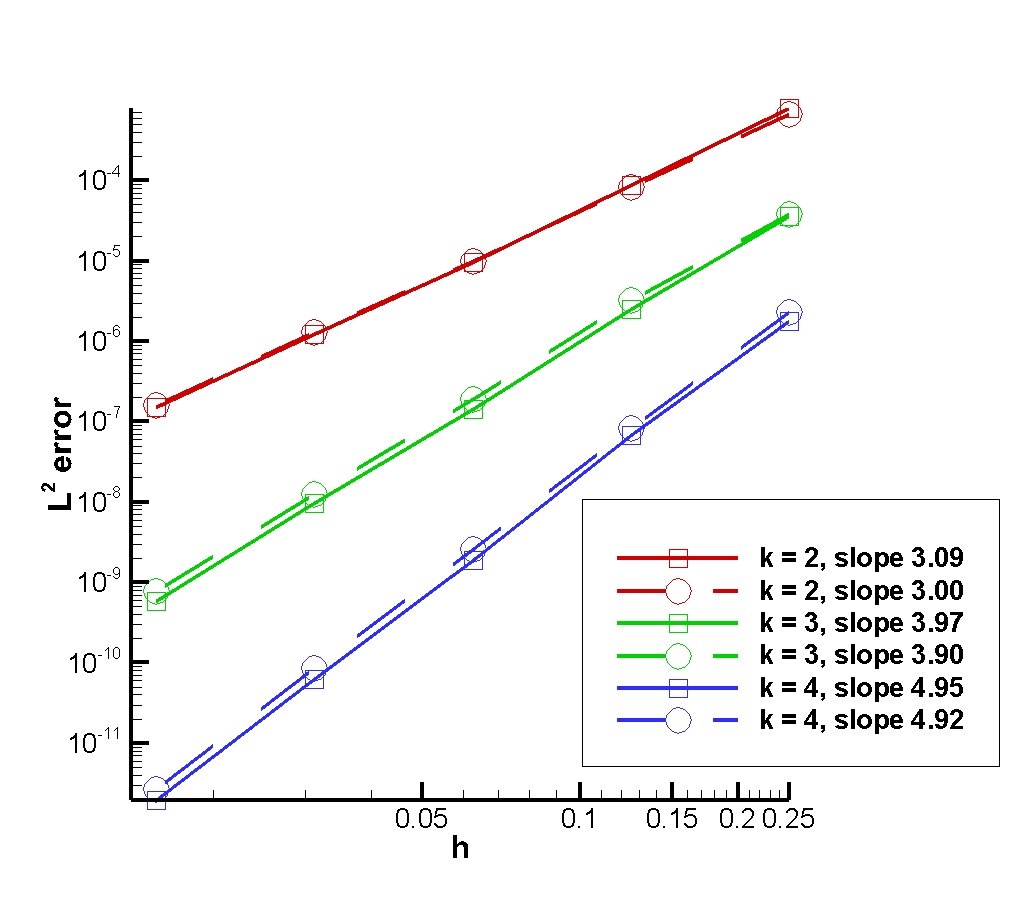}
\quad
\includegraphics[width=0.48\textwidth]{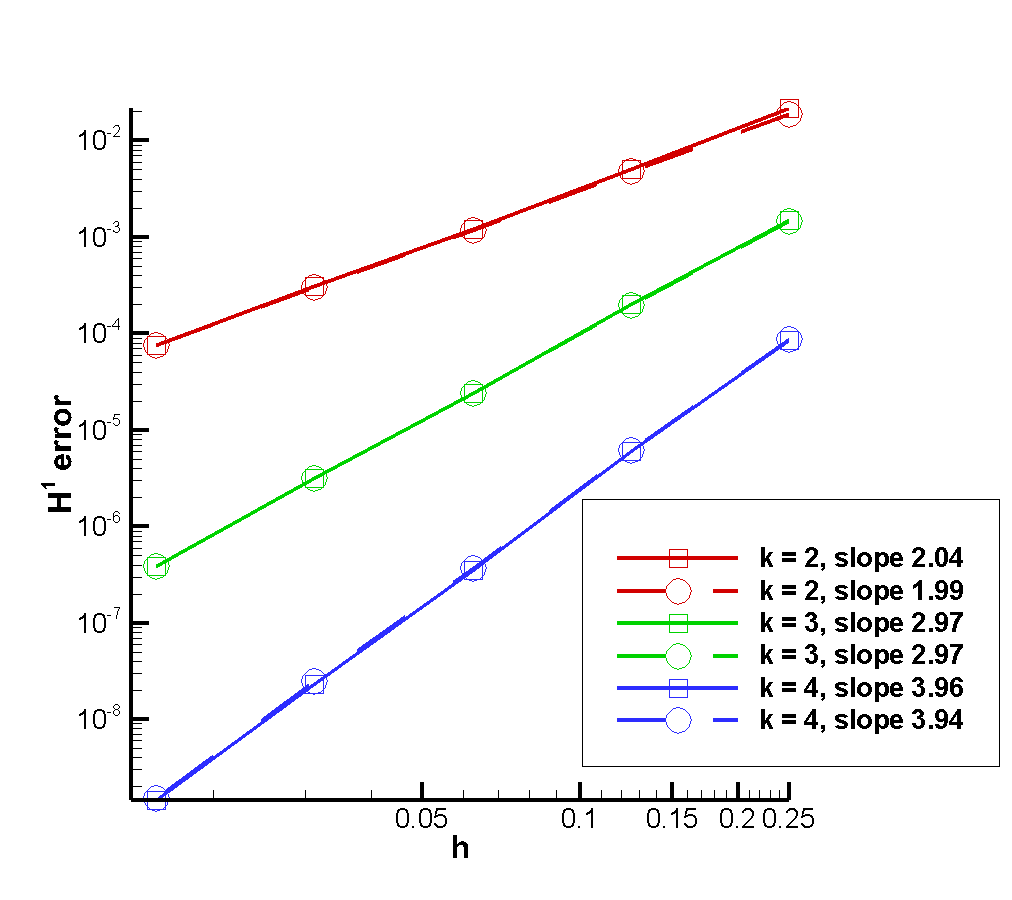}
}
\caption{Test case with the non constant coefficient matrix $A$: the error in $L^2$ norm and $H^1$ semi-norm vs. mesh-size $h$. The solid lines with squares represent the SIP method. The dashed lines with circles represent the scSIP method. }
\label{FigConv2}
\end{figure}

\new{We now consider problem (\ref{Pb}) with a non-constant coefficient matrix
\[ A = \left(\begin{array}{cc}
     1 + x & xy\\
     xy & 1 + y
   \end{array}\right) 
\]
set again on  $\Omega=(0,1)^2$. The right-hand side $f$ and non-homogeneous Dirichlet boundary conditions $g$ are chosen so that the exact solution is given by $u=e^{xy}$. The results are presented in Fig. \ref{FigConv2} using the same meshes and parameters $h$, $h_E$, and $\gamma$ as in the first test case. We arrive at the same conclusions about the convergence of SIP and scSIP methods as before.   }

\begin{acknowledgements}
I am grateful to Simon Lemaire for interesting discussions about HHO and msHHO methods, which were the starting point of conceiving the present article.
\end{acknowledgements}

\bibliographystyle{spmpsci}      
\bibliography{DG}   

\begin{thebibliography}{10}
\providecommand{\url}[1]{{#1}}
\providecommand{\urlprefix}{URL }
\expandafter\ifx\csname urlstyle\endcsname\relax
  \providecommand{\doi}[1]{DOI~\discretionary{}{}{}#1}\else
  \providecommand{\doi}{DOI~\discretionary{}{}{}\begingroup
  \urlstyle{rm}\Url}\fi

\bibitem{Adams}
Adams, R.A.: Sobolev spaces.
\newblock Academic Press [A subsidiary of Harcourt Brace Jovanovich,
  Publishers], New York-London (1975).
\newblock Pure and Applied Mathematics, Vol. 65

\bibitem{antonietti16}
Antonietti, P.F., Cangiani, A., Collis, J., Dong, Z., Georgoulis, E.H., Giani,
  S., Houston, P.: Review of discontinuous galerkin finite element methods for
  partial differential equations on complicated domains.
\newblock In \cite{bridges} (2016)

\bibitem{arnold82}
Arnold, D.N.: An interior penalty finite element method with discontinuous
  elements.
\newblock SIAM J. Numer. Anal. \textbf{19}(4), 742--760 (1982).
\newblock \urlprefix\url{https://doi.org/10.1137/0719052}

\bibitem{bridges}
Barrenechea, G.R., Brezzi, F., Cangiani, A., Georgoulis, E.H. (eds.): Building
  bridges: connections and challenges in modern approaches to numerical partial
  differential equations, \emph{Lecture Notes in Computational Science and
  Engineering}, vol. 114.
\newblock Springer, [Cham] (2016).
\newblock \urlprefix\url{https://doi.org/10.1007/978-3-319-41640-3}.
\newblock Selected papers from the 101st LMS-EPSRC Symposium held at Durham
  University, Durham, July 8--16, 2014

\bibitem{brezzi13}
Brezzi, F., Marini, L.D.: Virtual element methods for plate bending problems.
\newblock Comput. Methods Appl. Mech. Engrg. \textbf{253}, 455--462 (2013).
\newblock \urlprefix\url{https://doi.org/10.1016/j.cma.2012.09.012}

\bibitem{cangiani14}
Cangiani, A., Georgoulis, E.H., Houston, P.: {$hp$}-version discontinuous
  {G}alerkin methods on polygonal and polyhedral meshes.
\newblock Math. Models Methods Appl. Sci. \textbf{24}(10), 2009--2041 (2014).
\newblock \doi{10.1142/S0218202514500146}.
\newblock \urlprefix\url{https://doi.org/10.1142/S0218202514500146}

\bibitem{ciarlet91}
Ciarlet, P.G.: Basic error estimates for elliptic problems.
\newblock In: Handbook of numerical analysis, {V}ol. {II}, Handb. Numer. Anal.,
  II, pp. 17--351. North-Holland, Amsterdam (1991)

\bibitem{cockburn16}
Cockburn, B.: Static condensation, hybridization, and the devising of the {HDG}
  methods.
\newblock In: \cite{bridges} (2016)

\bibitem{cockburn_dipietro}
Cockburn, B., Di~Pietro, D.A., Ern, A.: Bridging the hybrid high-order and
  hybridizable discontinuous {G}alerkin methods.
\newblock ESAIM Math. Model. Numer. Anal. \textbf{50}(3), 635--650 (2016).
\newblock \urlprefix\url{https://doi.org/10.1051/m2an/2015051}

\bibitem{cockburn09}
Cockburn, B., Gopalakrishnan, J., Lazarov, R.: Unified hybridization of
  discontinuous {G}alerkin, mixed, and continuous {G}alerkin methods for second
  order elliptic problems.
\newblock SIAM J. Numer. Anal. \textbf{47}(2), 1319--1365 (2009).
\newblock \urlprefix\url{https://doi.org/10.1137/070706616}

\bibitem{dipietro_book}
Di~Pietro, D.A., Ern, A.: Mathematical aspects of discontinuous {G}alerkin
  methods, \emph{Math\'ematiques \& Applications (Berlin) [Mathematics \&
  Applications]}, vol.~69.
\newblock Springer, Heidelberg (2012).
\newblock \urlprefix\url{https://doi.org/10.1007/978-3-642-22980-0}

\bibitem{dipietro15}
Di~Pietro, D.A., Ern, A.: A hybrid high-order locking-free method for linear
  elasticity on general meshes.
\newblock Comput. Methods Appl. Mech. Engrg. \textbf{283}, 1--21 (2015).
\newblock \urlprefix\url{https://doi.org/10.1016/j.cma.2014.09.009}

\bibitem{dipietro14}
Di~Pietro, D.A., Ern, A., Lemaire, S.: An arbitrary-order and compact-stencil
  discretization of diffusion on general meshes based on local reconstruction
  operators.
\newblock Comput. Methods Appl. Math. \textbf{14}(4), 461--472 (2014).
\newblock \urlprefix\url{https://doi.org/10.1515/cmam-2014-0018}

\bibitem{dipietro16}
Di~Pietro, D.A., Ern, A., Lemaire, S.: A review of hybrid high-order methods:
  formulations, computational aspects, comparison with other methods.
\newblock In \cite{bridges} (2016)

\bibitem{ern}
Ern, A., Guermond, J.L.: Theory and practice of finite elements, \emph{Applied
  Mathematical Sciences}, vol. 159.
\newblock Springer-Verlag, New York (2004)

\bibitem{girault}
Girault, V., Raviart, P.A.: Finite element methods for {N}avier-{S}tokes
  equations, \emph{Springer Series in Computational Mathematics}, vol.~5.
\newblock Springer-Verlag, Berlin (1986).
\newblock \doi{10.1007/978-3-642-61623-5}.
\newblock \urlprefix\url{https://doi.org/10.1007/978-3-642-61623-5}.
\newblock Theory and algorithms

\bibitem{guyan65}
Guyan, R.J.: Reduction of stiffness and mass matrices.
\newblock AIAA journal \textbf{3}(2), 380 (1965)

\bibitem{freefem}
Hecht, F.: New development in freefem++.
\newblock J. Numer. Math. \textbf{20}(3-4), 251--265 (2012)

\bibitem{riviere_book}
Riviere, B.: Discontinuous Galerkin methods for solving elliptic and parabolic
  equations: theory and implementation.
\newblock SIAM (2008)

\bibitem{veiga13}
Beir\~ao~da Veiga, L., Brezzi, F., Marini, L.D.: Virtual elements for linear
  elasticity problems.
\newblock SIAM J. Numer. Anal. \textbf{51}(2), 794--812 (2013).
\newblock \urlprefix\url{https://doi.org/10.1137/120874746}

\bibitem{veiga16}
Beir\~ao~da Veiga, L., Brezzi, F., Marini, L.D., Russo, A.: Virtual element
  implementation for general elliptic equations.
\newblock In \cite{bridges} (2016)

\bibitem{wang13}
Wang, J., Ye, X.: A weak {G}alerkin finite element method for second-order
  elliptic problems.
\newblock J. Comput. Appl. Math. \textbf{241}, 103--115 (2013).
\newblock \urlprefix\url{https://doi.org/10.1016/j.cam.2012.10.003}

\bibitem{warburton03}
Warburton, T., Hesthaven, J.S.: On the constants in {$hp$}-finite element trace
  inverse inequalities.
\newblock Comput. Methods Appl. Mech. Engrg. \textbf{192}(25), 2765--2773
  (2003).
\newblock \doi{10.1016/S0045-7825(03)00294-9}.
\newblock \urlprefix\url{https://doi.org/10.1016/S0045-7825(03)00294-9}

\bibitem{wheeler78}
Wheeler, M.F.: An elliptic collocation-finite element method with interior
  penalties.
\newblock SIAM Journal on Numerical Analysis \textbf{15}(1), 152--161 (1978)

\end{thebibliography}

\end{document}